\documentclass[11pt]{article}
\usepackage{graphicx}
\usepackage{latexsym}
\usepackage{amsmath}
\usepackage{amssymb}
\usepackage{amsfonts}
\usepackage{enumerate}
\usepackage{natbib}
\def\cite{\citet}
\newtheorem{lemma}{Lemma}
\newtheorem{proposition}{Proposition}
\newtheorem{theorem}{Theorem}
\newtheorem{corollary}{Corollary}
\newcommand{\qed}{\hfill \mbox{\raggedright \rule{.07in}{.1in}}}
\newenvironment{proof}{\vspace{1ex}\noindent{\bf Proof}\hspace{0.5em}}
	{\hfill\qed\\[5pt]}%
\begin{document}
\title{Customer Equilibrium and Optimal Strategies \\in an
  $M/M/1$ Queue with Dynamic Service Control\footnote{Research
  supported by the Greek Secreteriat of Research and Technology under
  Greece/Turkey bilateral research collaboration
  program. G. Dimitrakopoulos also acknwledges support by the Athens
  University of Business and Technology under the Basic Research
  Funding Program.}}
\author{Yiannis Dimitrakopoulos\\
  Department of Informatics\\
  Athens University of Economics and Business\\
  76 Patission Str., Athens, Greece 10434\\
  email: dimgiannhs@aueb.gr \and
  Apostolos Burnetas\footnote{Corresponding Author}\\
  Department of Mathematics\\
  University of Athens\\
  Panepistemiopolis, Greece 15784\\
  email: aburnetas@math.uoa.gr } \date\today

\maketitle

\begin{abstract}
  We consider the problem of customer equilibrium strategies in an
  M/M/1 queue under dynamic service control.  The service rate
  switches between a low and a high value depending on system
  congestion.  Arriving customers do not observe the system state at
  the moment of arrival. We show that due to service rate variation,
  the customer equilibrium strategy is not generally unique, and
  derive an upper bound on the number of possible equilibria. For the
  problem of social welfare optimization, we numerically analyze the
  relationship between the optimal arrival rate, which maximizes the
  overall welfare of the customers, and the equilibrium ones as a
  function of various parameter values. We finally derive analytic
  solutions for the special case where the service rate switch occurs
  when the queue ceases to be empty.  \vskip0.5cm
\noindent {\bf Keywords:} Queueing; Customer Equilibrium Strategies;
Dynamic Service Control; Positive Externalities; Social Welfare
Maximization.
\end{abstract}

\section{Introduction}
\label{sec-Introduction}
\noindent
Service rate control is a prevalent tool for dynamically adjusting the
operation of a queueing system in response to varying congestion
levels. Many models in the queueing and service management literature,
such as dynamic staffing policies, server vacation mechanisms,
etc. can be viewed as different aspects of varying the rate at which
service is provided at different times. In most cases, the philosophy
of a service control policy is to increase the service capacity when
the congestion level is high in order to alleviate excessive delays,
and reduce it at times of decreased congestion in order to lower
operating costs. Viewed in this way, a service control policy provides
a balance in the tradeoff between long customer delays and high
service costs. However, in situations where customers are sensitive to
delay and make individual decisions about entering the queue, the
policy used by the service provider may also have an indirect, albeit
significant, effect on the arrival stream.

In service systems with strategic customer behavior,
incoming customers decide independently whether to join the system or
balk, based on an individual utility function that combines the value
of obtaining the service with the cost induced by the anticipated
delay in the queue. Since a customer's delay depends on the decisions
of other customers, a game-theoretic approach is appropriate for
characterizing customer behavior and thus the properties of the
incoming stream, as the result of an equilibrium strategy. Because
customers are averse to delay, it is expected that a particular
customer's decision to enter the system is adversely affected as other
customers also decide to enter and system congestion is increased. In
economic terms, entering customers induce negative externalities, by
imposing increased delays to present and/or future arrivals. Such
situations, where in equilibrium a customer's willingness to join is a
decreasing function of the arrival rate, are referred to as Avoid the
Crowd (ATC). On the other hand, in a system where a form of dynamic
service control is implemented, the service rate is usually increased
when the system is congested. Thus, it may happen that if a customer
decides to enter the system, the other customers experience a benefit,
because the service provider is induced to increase the service
rate. Then, the net effect between the additional delay due to
increased arrivals and the higher service speed due to the server's
policy may be positive. In such cases where an arriving customer
induces positive externalities, the equilibrium is Follow the Crowd
(FTC). Both the ATC or FTC effects may appear in the same system for
different values of the arrival rate. However in most situations, the
presence of FTC is related to a dynamic service policy employed by the
service provider.

In this paper we analyze the customer equilibrium behavior in a single server Markovian queue under a simple threshold-based service rate policy. Specifically, the service rate is kept at a low value when the number of customers in the system is at or below a threshold level $T$, and turns to a high value when the system congestion is above $T$. We analyze the unobservable case where arriving customers are aware of the service policy, but do not have any information on the queue length upon arrival.
The most pronounced effect of the service control policy is that, in contrast to the simple M/M/1 queue in strategic equilibrium, the equilibrium customer strategy is not generally unique. For example, one equilibrium strategy may be such that a small fraction of incoming customers joins the system and as a result the server works at slow rate for most of the time, whereas there exists another equilibrium where the arrival rate is high and the server spends most time at the fast mode.
We show that, depending on the values of the various parameters, the number of strategic equilibria may vary from one to at most three.
We also consider the problem of determining the optimal arrival rate in a system with centrally control arrivals which, for the given service policy, maximizes the expected net benefit of all customers, and compare the optimal arrival rates to those obtained as equilibria under selfish customer behavior.

\noindent The implications of strategic customer behavior on the performance of a queueing system has been studied extensively in the recent years. Early works on the M/M/1 queue include \cite{naor} and \cite{edelsonhildebrand} for the observable and unobservable models, respectively. Many variations of the original models have been studied since, and a comprehensive review of the literature until 2003 is provided in \cite{hassinhaviv}. Among the models that have been developed and analyzed, many include some varying service rate characteristics indirectly.
For example, the capacity choice problem can be viewed as long-term service rate decision. \cite{Stenbacka_Tombak1995} consider an unobservable system with nonhomogeneous  customer delay costs and service rate decisions made by the firm. \cite{Chen_Frank2004} analyze the joint impact of a firm's service and pricing policy on customer equilibrium.

In observable models, where customers obtain full or partial
information on system congestion or anticipated delay, a varying
service rate policy may have analogous effects. \cite{armonymaglaras}
analyze the impact of announcing anticipated delays and providing a
call back option on customer joining behavior, in a call center with
two service modes. \cite{Debo2009} consider a situation where the
service rate is related to the quality of service, it is unknown to
arriving customers but can be estimated by observing the queue length.
It is shown that when the service rate and the quality of service are negatively
correlated, customers join longer queues, which is not the case
under positive correlation.

\noindent  Other systems with a varying service rate feature include vacation queues and systems in a random environment.
\cite{burnetaseconomou} consider a system with server vacations and reactivation that requires a random setup time. In the vacation model of \cite{vacation_hass}, the server resumes service after a fixed number of arrivals according to a threshold service policy. Both papers analyze equilibrium behavior with respect to several levels of information about the queue length and the state of the server.
\cite{Economou_Manou2011} analyze customer equilibrium strategies in a stochastic clearing system where the exogenous customer arrival rate and the arrival rate of the system clearing server vary according to an external environment process.

In the area of queueing control and optimization under a single decision maker, problems of service rate control have been studied extensively and in many forms. It is worth noting that the problem of social welfare maximization can be formulated as a static or dynamic admission control problem. \cite{georgeharrison} consider welfare optimization in an $M/M/1$ system with dynamic service control and constant arrival rate. They develop an asymptotic method for computing the optimal policy under average cost minimization.
\cite{ata} propose a model that jointly optimizes arrival and service rates, and develop a dynamic pricing strategy that induces the optimal arrival and service rates. \cite{hasenbein2010} consider a model of joint admission and service control in an $M/M/1$ under an average cost criterion and show that the optimal service rates are increasing with system congestion. \cite{Dim_Burnetas2011} analyze the effect of service rate flexibility on the admission control policy and the service provider's profit and derive a condition for the service flexibility to be beneficial.

The present paper proposes a direct model of the impact of dynamic service control on customer equilibrium strategies.
The service rate is dynamically adjusted with state congestion.
Under this situation, the delay function is not monotone increasing in the arrival rate, which implies that for certain
ranges of the arrival rate, joining customers impose positive externalities. Another consequence is that in general there exist more than one strategic equilibria, however we show that their number is at most three. We also consider the social welfare optimization problem and the relationship between the equilibrium and optimal arrival rates. We show via computational experiments that  optimal arrival rate lies between the extreme values under equilibrium, due to the existence of both negative and positive externalities.

The rest of the paper is
organized as follows. In Section \ref{sec-Model} we introduce the
model and the corresponding customer strategic behavior problem.
We also derive the expected delay as a function of the arrival rate.
In Section \ref{sec-Equilibrium} we consider the equilibrium strategies and derive the upper bound on the number of equilibria for general
$T$-threshold dynamic service policies.
In Section \ref{sec-Social} we explore numerically the socially optimal strategy and the comparison with the equilibrium problem.
In Section \ref{sec-Special}, we derive analytic solutions for both problems
in the special case where the service threshold $T=1$, i.e., the service rate switches to the fast mode
when the queue is not empty. Section \ref{sec-Summary} concludes.

\section{Model Description}\label{sec-Model}
\noindent We consider a single server Markovian queue under the FCFS
discipline, where potential customers arrive according to a Poisson
process with rate $\Lambda$. The service rate is dynamically
set according to a
$T$-threshold policy. Specifically, the service rate
is set to $\mu_l$ whenever the number of customers in the system is
less than or equal to $T$, and equal to $\mu_h$ otherwise, where $\mu_l<\mu_h$. There are no switching costs. Thus, the service policy is defined by the parameters $(T,\mu_l,\mu_h)$, which are known to all arriving customers.

Arriving customers are assumed identical. They make the decision whether to join or balk upon arrival, in order to maximize their expected net benefit, thus they are assumed risk neutral. Every joining customer receives a fixed
reward $R\geq0$ upon service completion, and incurs a service cost
rate $C>0$ per time unit of delay. In addition, each arriving customer has no information for the
queue length prior to entering the system.

 Since the join decisions of individual customers affect the system delay and thus the benefit of all customers, the decision problem corresponds to a symmetric game. We restrict attention to symmetric Nash equilibrium strategies defined as follows.

 Given that potential customers are not aware of the actual system state, an arriving customer has two pure strategies, either to join the system or balk. However, for the equilibrium analysis, mixed strategies
also have to be considered. A mixed strategy is defined by
$p\in[0,1]$, the probability of joining the system, where for
$p=0,\;\hbox{or}\;1$ it reduces to a pure strategy.

 Let $U(p;q)$ denote a tagged customer's expected net benefit
from using mixed strategy $p$, when all others adopt mixed
strategy $q$. Then, a mixed strategy $q^e$ is a symmetric Nash
equilibrium strategy, if $U(q^e,q^e)\geq U(q,q^e)$ for any mixed
strategy $q\in[0,1]$. The intuition is that an equilibrium strategy is the best
response against itself, so if all customers agree to follow mixed
strategy $q^e$, no one can benefit
from changing it.

Under a mixed strategy $q$ arrivals follow a Poisson process with
effective arrival rate $\lambda=\Lambda q$. The corresponding equilibrium arrival rate is denoted by $\lambda^e=\Lambda\cdot q^e$, where $q^e$ is a symmetric equilibrium strategy.
In addition, the customer's expected net benefit is equal to
\begin{equation}\label{modexpwt}U(p;q)=p\left[R-CW(\Lambda q)\right],\end{equation}
\noindent where $W(\lambda)$ denotes the total
sojourn time of a customer in the system, as a function of the
effective arrival rate $\lambda$, assuming $\lambda<\mu_h$, in order to ensure
stability.

 In the remainder of the paper, we refer to
$W(\lambda)$ as the waiting time or the delay function and, without loss of generality, we normalize $\mu_h=1$ and $C=1$.

 An alternative decision framework could be defined by assuming that all customer decisions are controlled by a central planner, who enforces a mixed strategy $q$ with the objective of maximizing the overall expected net benefit per unit time. This corresponds to the social welfare optimization problem.

 Given that a common strategy $q$ corresponds to a Poisson arrival rate $\lambda=\Lambda q$, the welfare function can be defined in terms of $\lambda$ as
\begin{equation}\label{socben}
S(\lambda)=\lambda\left[R-W(\lambda)\right]\;\hbox{for}\;\lambda\in[0,1)
\end{equation}
\noindent and an arrival rate $\lambda^*$ is socially optimal if
\begin{equation}\label{socoptpr}S(\lambda^*)=\max_{\lambda\in[0,1)} S(\lambda).\end{equation}
\noindent In general the equilibrium and optimal arrival rates rates do not coincide.

 In the following sections we analyze the equilibrium and social optimization problems and their relationship.

 In the remainder of this section we derive the expression for the delay function $W(\lambda)$.

 An
unobservable $M|M|1$ queue under a T-threshold service policy and a common mixed strategy $q$ for all customers
corresponds to a standard birth and death process with arrival rates
$\lambda_n=\lambda=\Lambda q,\;n\in{\bf{N}}$ and departure rates
$\mu_n=\left\{\begin{array}{ll} \mu_l,&n\leq T\\
1,&n>T\end{array}\right.,\;n\in{\bf{N^*}},$
where $n$ refers to the number of customers in the system.

 Using standard queueing analysis, we derive the following
expression for the waiting time
\begin{equation}\label{expwt}
W(\lambda)=\frac
{\frac{1}{\mu_l-\lambda}\{\mu_l^T-(T+1)\lambda^T+\frac{\lambda}{\mu_l-\lambda}(\mu_l^T-\lambda^T)\}+\frac{\lambda^T(T+1-T\lambda)}{(1-\lambda)^2}}
{\frac{\mu_l^{T+1}-\lambda^{T+1}}{\mu_l-\lambda}+\frac{\lambda^{T+1}}{1-\lambda}},
\end{equation}
\noindent for $\lambda\in[0,\mu_l)\cup(\mu_l,1)$.

 An analogous expression of $W(\lambda)$ can be obtained
for the case $\lambda=\mu_l$. After simplification
both expressions can be reduced to the following, in which the
denominator is strictly positive in
$\lambda\in[0,1)$:
\begin{equation}
W(\lambda)=\frac{(\mu_l-\lambda)^2g(\lambda)}{(1-\lambda)(\mu_l-\lambda)^2d(\lambda)}=\frac{g(\lambda)}{(1-\lambda)d(\lambda)},\;\hbox{for
any}\;\lambda\in[0,1), \label{simexpwt}
\end{equation}
\noindent where
\begin{eqnarray}\label{simnumewt}
\nonumber g(\lambda)&=&-(T-1)(1-\mu_l)\lambda^{T}\\&&-(1-\mu_l)\sum_{j=1}^{T-1}\mu^{j-1}_l\left[(T-(j+1))\mu_l+j-1-T\right]\lambda^{T-j}+\mu^{T-1}_l
\end{eqnarray}
\noindent and
\begin{equation}\label{simdenewt}
d(\lambda)=(1-\mu_l)\sum_{j=0}^{T-1}\mu^{j}_l\lambda^{T-j}+\mu^{T}_l.
\end{equation}
\noindent Note that, $W(0)=\frac{1}{\mu_l}>0$ and $\lim_{\lambda\rightarrow1^-}W(\lambda)=+\infty$.

 In addition, using \eqref{simexpwt} we obtain a lower
bound for the waiting time:
\begin{equation}\label{boundwt} W(\lambda)>\frac{1}{1-\lambda}, \end{equation}
where the right-hand side is the
expected waiting time in an $M|M|1$ queue
with service rate equal to $1$.

\section{Equilibrium Analysis} \label{sec-Equilibrium}
\noindent In this section, we identify the
equilibrium strategies under individual customer behavior.

 Using the standard methodology of equilibrium analysis in

 unobservable queueing models (see \cite{hassinhaviv}
p. 46-47), we derive from \eqref{modexpwt} the following equilibria:
\begin{equation}\label{equil0}
\lambda^e=0\;\hbox{if and only if}\;W(0)=\frac{1}{\mu_l}\geq R.
\end{equation}
\noindent Note that, when $R<1$, $0$ is the unique symmetric
equilibrium, because of \eqref{boundwt}. Similarly,
\begin{equation}\label{equilL}
\lambda^e=\Lambda\;\hbox{if and only if}\;W(\Lambda)<R.
\end{equation}
\noindent In addition, a necessary and sufficient condition, for
$\lambda^e=\Lambda q^e\in(0,1)$ to be an equilibrium arrival rate is that
\begin{equation}\label{condequil} U(q^e;\lambda^e)=0,\;\hbox{or equivalently,}\;W(\lambda^e)=R.
\end{equation}
\noindent Thus, the problem of identifying symmetric equilibria other than the extreme $\lambda^e=0$ or $\Lambda$, is equivalent to
finding the solutions of the equation
\begin{equation}\label{rationaleq}
W(\lambda)=R,\;\hbox{in}\;\lambda\in(0,1),\;\hbox{for}\;R\geq1.
\end{equation}
\noindent In the following we assume without loss of generality that the external arrival rate
$\Lambda\geq1$, in order to identify all possible equilibrium
arrival rates $\lambda^e$ in the interval $[0,1)$.
In the case where $\Lambda<1$, some of these equilibria may be excluded and instead we must also consider $\Lambda$ as a possible equilibrium arrival rate.

 From the properties of the delay function $W(\lambda)$ we derive the following lemma.
\begin{lemma}\label{leastroot}
\noindent If $R>\frac{1}{\mu_l}$, then the equation $W(\lambda)=R$ has at least one solution in $\lambda\in(0,1)$.
\end{lemma}
\begin{proof}
\noindent  From the divergence of $W(\lambda)$ to infinity as $\lambda\rightarrow 1^{-}$, we derive that for any $R\geq1$, there exists an arrival rate $\lambda_0\in(0,1)$ such that $W(\lambda_0)\geq R$.

 In addition, $W(0)=\frac{1}{\mu_l}<R$ and the result follows readily from the continuity of $W(\lambda)$ in $\lambda\in[0,\lambda_0]$.\
\end{proof}
\begin{figure}
\centering
\includegraphics[width=0.9\textwidth]{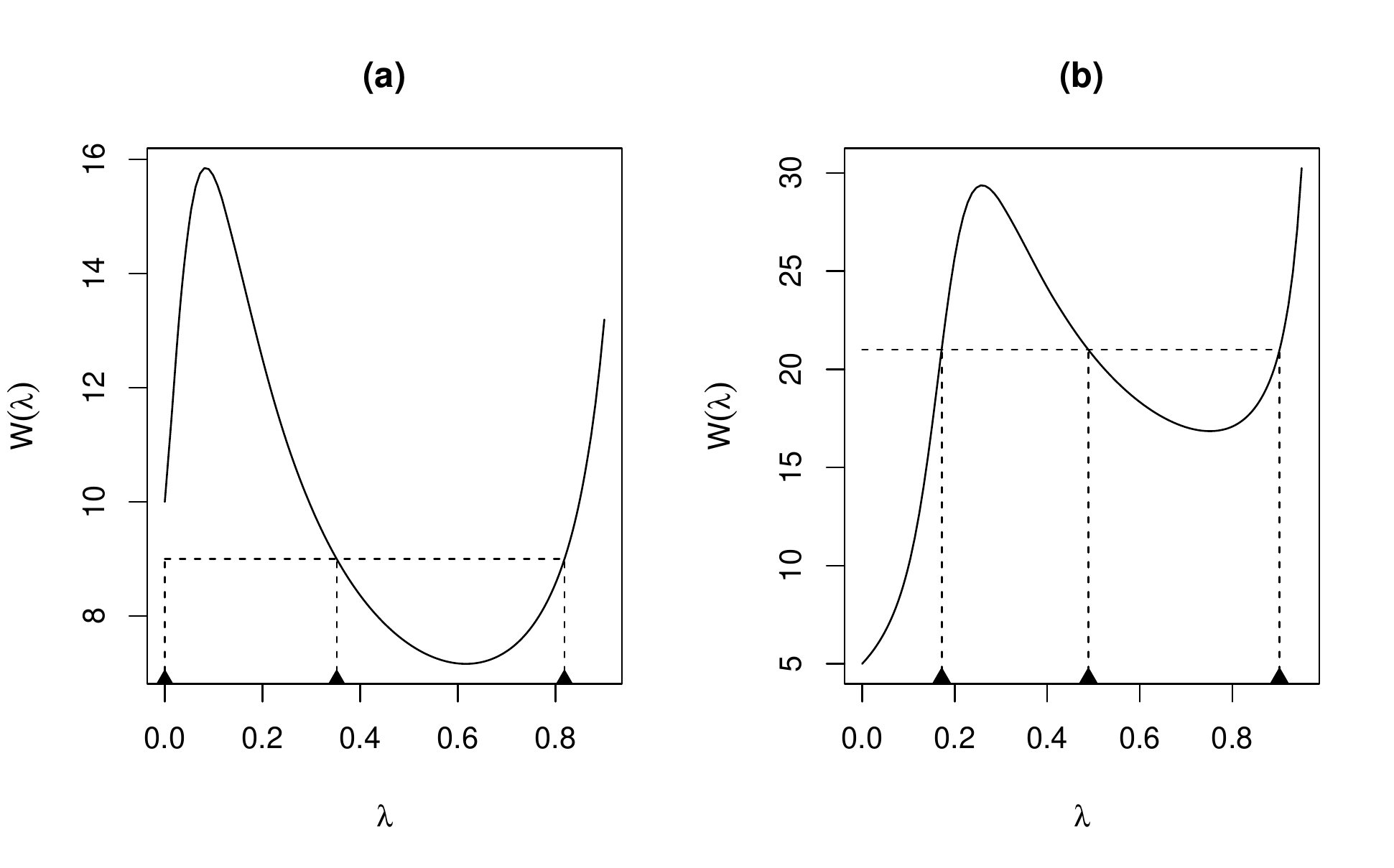}
\caption{The expected waiting time, $W(\lambda)$, for
(a)$T=3,\;\mu_l=\frac{1}{10}$ and (b)$T=10,\;\mu_l=\frac{1}{5}$}\label{wplots}
\end{figure}
\noindent From Lemma \ref{leastroot}, it follows that for $R>\frac{1}{\mu_l}$ there exists at least one positive equilibrium arrival rate. On the other hand, extensive numerical analysis
for different values of the parameters $R,\;T$ and $\mu_l$ has shown that there are at most three symmetric equilibria $\lambda^e\in(0,1)$.

 This observation is consistent with the behavior of the delay function $W(\lambda)$, as demonstrated in Figure \ref{wplots} for two typical cases of $T$-threshold policies $(T,\mu_l,1)$. In both
plots, we observe that $W(\lambda)$ increases for low and high values of $\lambda$, whereas for intermediate values of $\lambda$, it decreases. Intuitively, this is happening because for small and large values of $\lambda$, the system spends almost all time in the low or high service rate state, respectively, and the delay is approximately that of an $M|M|1$ queue with the corresponding value of $\mu$.

 On the other hand, for intermediate values of $\lambda$, the service rate switch occurs more often. In this case, increasing $\lambda$ increases the rate of arrival to the queue, but at the same time forces the server to work faster and the net effect may be a decrease in the overall delay.

 Because of the non-monotone behavior of $W(\lambda)$, the number of roots of \eqref{rationaleq} may vary from $0$ to $3$, depending on the values of $R,\;T$ and $\mu_l$.

 For example in Figure \ref{wplots}(a), for $R=9$, there are three equilibrium arrival rates, one of which equal to zero, while in Figure \ref{wplots}(b) for $R=21$, there are three strictly positive equilibria.

 Note that this discussion is solely based on observation of the delay function graph for different numerical values. The complexity of the expressions does not allow the study of $W(\lambda)$ directly.

 We instead resort to the following equivalent polynomial equation of degree $T+1$, derived by substituting \eqref{simexpwt} in \eqref{rationaleq}.
\begin{eqnarray}\label{poleq}
\nonumber H(\lambda)&=&-R(1-\mu_l)\lambda^{T+1}+(1-\mu_l)(R(1-\mu_l)+T-1)\lambda^{T}\\&&+(1-\mu_l)\sum_{j=2}^{T-1}\mu_l^{T-1-j}\left[-R\mu^2+(R+j-1)\mu-(j+1)\right]\lambda^j\\
\nonumber
&&+\mu_l^{T-2}(-2R\mu_l^2+(R+2)\mu_l-2)\lambda+\mu_l^{T-1}(R\mu_l-1)=0.
\end{eqnarray}
\noindent In general, the roots of \eqref{poleq} cannot be analytically expressed as a
function of the parameters $T,\;\mu_l$ and $R$. Furthermore, the number of roots and thus of equilibria can be as high as $T+1$. However we apply analytic techniques to bound the number of roots.

 Specifically, we will make use of the following result in analysis, referred to in the literature as
Budan's Rule.
\begin{theorem}({\it{Budan's Rule, \cite{wilf} p.94}})\label{budanrule}

 The number of zeros of a polynomial equation
\begin{equation*}
f(x)=\alpha_0+\alpha_1x+\alpha_2x^2+\cdots+\alpha_nx^n
\end{equation*}
\noindent in $(a,b)$ is either equal to $V(a)-V(b)$ or less by an
even number, where $V(x)$ is the number of variations of sign in the
sequence $(f(x),f'(x),\cdots f^{(n)}(x))$.
\end{theorem}
\noindent In our case, let $V(\lambda)$ be the number of sign
variations in the sequence

 $(H(\lambda),H'(\lambda),\cdots,
H^{(T+1)}(\lambda))$. Based on this theorem we can show the following result.
\begin{proposition}\label{rootsofpol}
For any $R\geq1$,
\begin{itemize}
\item[i.] $V(0)\leq3$
\item[ii.] The polynomial equation $H(\lambda)=0$ has at most three
roots in $\lambda\in(0,1)$.
\end{itemize}
\end{proposition}

\begin{proof}
\noindent{\it{i.}} By definition, $V(0)$ is the number of variations
of sign in the sequence

 $(H(0),H'(0),\ldots,H^{(n)}(0),\ldots,H^{(T+1)}(0))$, where
$H^{(n)}(0)=n!\eta_{n}\;\hbox{for}\; n=0,\ldots,T+1$
\noindent and $\eta_{n}$ denotes the coefficient of $\lambda^{n}$ in
polynomial $H(\lambda)$.

 Since $R\geq1$ and $\mu_l\in(0,1)$, we obtain the
following.

 First,
\begin{equation}
\label{signcfT+1}\eta_{T+1}=-R(1-\mu_l)<0,\;\hbox{therefore}\;
H^{(T+1)}(0)=(T+1)!\eta_{T+1}<0.
\end{equation}
\noindent Also,
\begin{equation}
\label{signcfT}\eta_{T}=(1-\mu_l)(R(1-\mu_l)+T-1)>0,\;\hbox{and
thus}\; H^{(T)}(0)=T!\eta_{T}>0.
\end{equation}
\noindent In order to derive the sign of coefficients $\eta_j$ for
$j=2,\ldots,T-1$, we let
\begin{equation}\label{pcfj}f(j)=-R\mu_l^2+(R-1)\mu_l-1-(1-\mu_l)j\end{equation}
\noindent and rewrite $H^{(j)}(0)$ as
\begin{equation}\label{cf2..T-1} H^{(j)}(0)=j!(1-\mu_l)\mu_l^{T-1-j}f(j).\end{equation}
\noindent Since $f(j)$ is decreasing in
$j$, there exists at most one sign variation in
$f(j)$ for $j\in\{2\ldots T-1\}$. Thus, there exists $j_0\in\{2\ldots T\}$,
such that for any $j=2,\ldots,T$, $f(j)\geq0$ if $j<j_0$ and $f(j)<0$ if $j\geq j_0$.
In the following sign analysis of $H^{(j)}(0)$, $0$ will be considered as positive.
 We consider three cases for $j_0$:
\begin{itemize}
\item[$1.$]If $j_0=2$, then $H^{(j)}(0)<0$
for any $j\in\{2,\ldots,T-1\}$.
\item[$2.$]If $j_0=T$, then $H^{(j)}(0)>0$
for any $j\in\{2,\ldots,T-1\}$.
\item[$3.$] If $j_0\in\{3\ldots,T-1\}$, then
$H^{(j)}(0)>0$ for $j<j_0$ and $H^{(j)}(0)<0$ for $j\geq j_0$.
\end{itemize}
\noindent For each of these, there are four subcases depending on the
 signs of $H'(0)$ and $H(0)$:
\begin{itemize}
\item[a.] $H'(0)>0$ and $H(0)>0$.
\item[b.] $H'(0)>0$ and $H(0)<0$.
\item[c.] $H'(0)<0$ and $H(0)>0$.
\item[d.] $H'(0)<0$ and $H(0)<0$.
\end{itemize}
\noindent In Table \ref{tableofsigns}, we summarize the sign
analysis of $H^{(j)}(0)$, for $j=0,\ldots,T+1$, and the value of $V(0)$ for each one combination.\\
\begin{table}
\begin{center}
\begin{tabular}{|c|cccc|cccc|cccc|}
\hline
 &\multicolumn{4}{|c|}{Case 1:$j_0=2$}&\multicolumn{4}{|c|}{Case 2:$j_0=T$}&\multicolumn{4}{|c|}{Case 3:$j_0\in\{3,\ldots,T-1\}$}\\
 \hline
 &a&b&c&d&a&b&c&d&a&b&c&d\\
\hline
   $H^{(T+1)}(0)$& $-$& $-$& $-$&  $-$& $-$& $-$& $-$& $-$ & $-$& $-$& $-$& $-$   \\
    $H^{(T)}(0)$& $+$& $+$& $+$& $+$& $+$& $+$& $+$& $+$ & $+$& $+$& $+$& $+$   \\
    $H^{(T-1)}(0)$&  $-$&  $-$& $-$& $-$&  $+$& $+$& $+$& $+$ & $+$& $+$& $+$& $+$   \\
    \vdots&   \vdots&    \vdots&   \vdots&  \vdots&    \vdots&    \vdots&   \vdots&  \vdots&    \vdots&    \vdots&   \vdots&  \vdots  \\
    $H^{(j_0-1)}(0)$& $-$&  $-$& $-$& $-$& $+$& $+$& $+$& $+$ & $+$& $+$& $+$& $+$  \\
    $H^{(j_0)}(0)$&   $-$&  $-$& $-$& $-$& $+$& $+$& $+$& $+$ & $-$& $-$& $-$& $-$\\
    \vdots&   \vdots&    \vdots&   \vdots&  \vdots&    \vdots&    \vdots&   \vdots&  \vdots&    \vdots&    \vdots&   \vdots&  \vdots    \\
    $H''(0)$&   $-$&  $-$& $-$& $-$&  $+$& $+$& $+$& $+$ & $-$& $-$& $-$& $-$  \\
    $H'(0)$&  $+$& $+$& $-$&  $-$&   $+$& $+$& $-$&  $-$ & $+$& $+$& $-$&  $-$ \\
    $H(0)$&   $+$& $-$&   $+$& $-$&    $+$& $-$&   $+$& $-$ &    $+$& $-$&   $+$& $-$ \\
\hline
    $V(0)$&3&4&3&2&1&2&3&2&3&4&3&2\\
\hline
\end{tabular}
\end{center}
\caption{Signs of $H^{(j)}(0)$ for $j=0,\ldots,T+1$ and the value of
$V(0)$.}\label{tableofsigns}
\end{table}
\noindent It follows that $V(0)\leq3$ is always true, except for two situations where subcase b holds.

 Thus to prove the Proposition, it suffices to show that subcase b cannot be true.

 Suppose that $H'(0)>0$ and $H(0)<0$. From \eqref{poleq}, we obtain that
\begin{equation*}
-2R\mu_l^2+(R+2)\mu_l-2>0\;\hbox{ and }\;R\mu_l-1<0.
\end{equation*}
\noindent However,
\begin{equation*}
-2R\mu_l^2+(R+2)\mu_l-2=(R\mu_l-1)-(2R\mu_l^2+1-\mu_l)<R\mu_l-1,
\end{equation*}
\noindent which is a contradiction. Thus subcase b is excluded and $V(0)\leq3$.

{\it{ii.}} Since $V(1)\geq0$, we obtain that
$V(0)-V(1)\leq3$, thus the result follows readily from Theorem \ref{budanrule}.\
\end{proof}
\noindent We can further observe from the Proof of Proposition \ref{rootsofpol}, that if subcase d holds, then $V(0)\leq2$,
thus, the upper bound on the number of roots of $H(\lambda)$ is reduced by one.

 In the next Corollary we provide sufficient condition for this sharper bound to hold.
\begin{corollary}\label{Cormaxeq2}
\noindent If $1\leq R<\frac{1}{\mu_l}$, then
$H(\lambda)$ has at most two roots in $(0,1)$.
\end{corollary}
\begin{proof}
\noindent If $1\leq R<\frac{1}{\mu_l}$, it follows that
$\eta_0=\mu_l^{T-1}(R\mu_l-1)<0$ and
$\eta_1=\mu_l^{T-2}(-2R\mu_l^2+(R+2)\mu_l-2)<0$. Thus, subcase d
holds and $V(0)=2$.\
\end{proof}
\noindent From Corollary \ref{Cormaxeq2} and
Proposition \ref{rootsofpol} it follows that the equation
$H(\lambda)=0$ has at most three roots in $(0,1)$, for $R\geq\frac{1}{\mu_l}$.

 In the following Theorem we
summarize the equilibrium analysis of our model.

 The proof is immediate from \eqref{equil0}, \eqref{equilL}, Lemma \ref{leastroot}, Proposition \ref{rootsofpol} and Corollary \ref{Cormaxeq2}.
\begin{theorem}\label{equilnum} For $\Lambda\geq1$:
\begin{itemize}
\item[i.] If $R<1$, the unique equilibrium
arrival rate is $\lambda^e=0$.
\item[ii.] If $1\leq R\leq\frac{1}{\mu_l}$, then $\lambda^e=0$ and in addition there may exist at most two positive equilibria.
\item[iii.] If $R>\frac{1}{\mu_l}$, then there exist at least one and at most three positive equilibria.
\end{itemize}
\end{theorem}
\noindent The existence of more than one symmetric equilibria is due to
the non-monotone behavior of the delay function $W(\lambda)$, as discussed previously.

 Also, it is worth noting that when multiple equilibria exist, then some
of them are stable, in the sense that a small perturbation of $\lambda$ will not affect
system convergence to the equilibrium rate, while others are not.

 For example, in Figure \ref{wplots}(a), there exist three
equilibrium arrival rates, of which $0$ and $\lambda^e_2$ are stable.
On the other hand, $\lambda^e_1$ is unstable, because a small increase
in $\lambda$ decreases the expected waiting time and more
customers are encouraged to join. This further increases the
arrival rate, thus the arrival rate moves away from $\lambda^e_1$. Similar
observations can be made for the example illustrated in Figure
\ref{wplots}(b). In this
case there exist three strictly positive equilibria, with
$\lambda^e_2$ being unstable, as it corresponds to the decreasing part of
$W(\lambda)$. Finally, in both examples, for values of
$R$ greater than the minimum value of $W(\lambda)$, the maximum
equilibrium arrival rate is always
stable.\\

\section{Social Optimization}\label{sec-Social}

\noindent In this section, we analyze the problem of
social welfare optimization, where the decision problem for a central planner
is to obtain a socially optimal
strategy $\lambda^*$, which maximizes customers' overall expected net benefit per unit time $S(\lambda)$, as defined in \eqref{socoptpr}.

 From \eqref{simexpwt}, the social welfare function can be expressed as
\begin{equation}\label{subsocben}
S(\lambda)=\lambda\left[R-\frac{g(\lambda)}{(1-\lambda)d(\lambda)}\right],\;\hbox{for}\;\lambda\in[0,1),
\end{equation}
\noindent where $g(\lambda)$ and $d(\lambda)$ are defined in
\eqref{simnumewt} and \eqref{simdenewt}, respectively.

 From \eqref{subsocben}, $S(\lambda)$ is continuous
and differentiable in $(0,1)$, thus, the first-order condition for the
maximization problem is
\begin{equation}\label{focsoc}
S'(\lambda)=0\Leftrightarrow R-W(\lambda)-\lambda W'(\lambda)=0.
\end{equation}
\noindent An analytic characterization of the optimal solution based on \eqref{focsoc} is not tractable because of the complexity of the expressions of
$W(\lambda)$ and $W'(\lambda)$. We thus resort to numerical analysis in order to explore the behavior the optimal arrival rate $\lambda^*$ and its relationship to the equilibrium arrival rates.\\
More specifically, we examine how the equilibrium and optimal arrival rates are affected by changing the values of parameters $T$, $\mu_l$ and $R$. Since, in general, the equilibrium arrival rate is not unique, in the diagrams we show the minimum and maximum equilibrium arrival rates, denoted as $\lambda^e_{\min}$ and $\lambda^e_{\max}$ respectively, and where the socially optimal arrival rate lies with respect to the range of individual equilibria. In addition, we examine the behavior of the maximum social benefit $S(\lambda^*)$ with respect to the same parameters.

 In the beginning, we present numerical results with respect to the service threshold $T$, which is an easily controlled operational characteristic of the system. We next proceed to the sensitivity analysis with respect to parameters $R$ and $\mu_l$, for specific values of the service threshold. We consider the cases $T=10$ and $T=3$, the same as those used for the graphs of the delay function $W(\lambda)$, in Figure \ref{wplots}.

 In Figure \ref{Tsoceq} we present the results of the sensitivity analysis with respect to the service threshold $T$, for $\mu_l=\frac{1}{4}$ and $R=25$. We observe that the equilibrium and optimal arrival rates are decreasing in $T$. This is expected because higher values of the service threshold increase the time that the server spends in the low rate and customers decision to join leads to increased delays due to system congestion, discouraging more customers from joining.

 Moreover, for low and high values of $T$, the system behaves approximately as an $M|M|1$ queue with service rate equal to $\mu_l$, and $1$, respectively. Therefore, for such $T$ values the equilibrium arrival rate is unique and the optimal arrival rate is lower than the equilibrium because of the negative externalities. On the other hand, for intermediate values of $T$ there are multiple equilibria, as discussed in Section \ref{sec-Equilibrium}. In these cases the optimal rate is between the two extreme values $\lambda^e_{\min}$ and $\lambda^e_{\max}$ of the equilibrium. Thus, whether the externalities are positive or negative, and thus whether a tax or a subsidy is required to induce efficiency, is not known in advance but it depends on which of the equilibrium rates will materialize. Finally, we observe in Figure \ref{Tsoceq} that the maximum social benefit $S(\lambda^*)$ is decreasing in $T$ for fixed values of $R$ and $\mu_l$, as a consequence of the behavior of the optimal arrival rate with respect to $T$.

 Next, we proceed to numerical results with respect to parameters $R$ and $\mu_l$, considering first the case of
$T=10$. Using the dynamic service policy $(T,\mu_l,1)$ for $T=10$ and $\mu_l=\frac{1}{5}$, we obtain $\lambda^*$ and the corresponding equilibria with respect to $R$, which are illustrated in Figure \ref{RT10soceq}. Note that, the equilibrium analysis for this service policy has been presented in Figure \ref{wplots}(b) and, thus, both diagrams must be considered together, in order to explore the relationship between $\lambda^*$ and the equilibria. We see that, $\lambda^*\leq\lambda^e_{\max}$ for low and high values of the reward, in which case the equilibrium arrival rate is unique, as showed in Figure \ref{wplots}(b). On the other hand, for intermediate values of $R$, where there exist two or three equilibria, we observe that $\lambda^e_{\min}\leq\lambda^*\leq\lambda^e_{\max}$. In this case, customers can impose either positive or negative externalities in equilibrium, as we mentioned before.

 Regarding the sensitivity of the social optimal arrival rate and the corresponding equilibria with respect to the low service rate for $R=20$, which is presented in Figure \ref{T10figsoc}, we see that $\lambda^e_{\min}\leq\lambda^*\leq\lambda^e_{\max}$ for low values of $\mu_l$, whereas $\lambda^*\leq\lambda^e_{\max}$ for higher values. This happens because the service rate switch option is no longer beneficial for joining customers, as the low service rate increases to high values. Therefore, the expected waiting time becomes monotone increasing as $\mu_l$ increases, corresponding to the delay function of a typical $M|M|1$ queue, which results in the existence of a unique equilibrium strategy with rate larger than $\lambda^*$.

 Moreover in both figures, we observe that there exist a discontinuity point with respect to $R$ or $\mu_l$, which corresponds to the existence of two equal maxima for different values of the reward or the low service rate, respectively. This is due to the fact that the social benefit function $S(\lambda)$ can be bimodal for some values of $R$ or $\mu_l$.

 Considering the behavior of $\lambda^*$ as well as both equilibria, we observe that they are increasing in $R$, which is intuitive considering the fact that customers either as a whole or acting independently are generally better off as the reward increases, and, thus tend to join more. This is also reflected in the behavior of the maximum social benefit $S(\lambda^*)$, which is always increasing in $R$, as observed in Figure \ref{RT10soceq}.

 In contrast, regarding the behavior of the social optimal arrival rate with respect to $\mu_l$, we observe in Figure \ref{T10figsoc} that it is slightly decreasing for small values, and, then becomes increasing, after a discontinuity jump. For this set of values of $T,\;R,\;\mu_l$ the social benefit function is bimodal in $\lambda$, with two local maxima. Of these the lower corresponds to a policy that accepts a small proportion of customers and thus keeps delay low by avoiding congestion.

 The other local maximum occurs at a higher value of $\lambda$, which means that more customers are accepted, so the server is forced to work at the high rate most of the time, thus keeping the delay moderate. For every value of $\mu_l$ one of the two local maxima supersedes the other and determines the optimal value of $\lambda$. For small values of $\mu_l$ the lower mode value is so low that almost no one is accepted, thus the higher mode is preferred. On the other hand as $\mu_l$ increases, the delay becomes more moderate even when the server works at low rate, which makes the lower mode preferred. The discontinuity point corresponds to the value of $\mu_l$ where the two local maxima are equal. As $\mu_l$ increases above the discontinuity point, the low mode still persists but nevertheless more customers are accepted under the optimal policy.

 Note that regardless of the non-monotone behavior of $\lambda^*$, the optimal welfare value $S(\lambda^*)$ is increasing in $\mu_l$, since higher values of this service rate can only improve system performance.

 Similar observations can be made for the case $T=3$. For $\mu_l=\frac{1}{10}$ and varying $R$, we observe in Figure \ref{soceqRT3} that $\lambda^*\leq\lambda^e_{\max}$ for $R<7$ and $R>16$, whereas
$\lambda^e_{\min}\leq\lambda^*\leq\lambda^e_{\max}$ for intermediate values of $R$. Moreover, regarding the behavior of the social
optimal arrival rate, $\lambda^*$ has a discontinuity point in
$R=7$. This can be explained by an analogous bimodality argument as in the previous case. In addition, the social optimal arrival rate $\lambda^*$, as well as the social optimal benefit $S(\lambda^*)$, are increasing in
$R$ as expected.

 We finally note that the discontinuities discussed above do not always occur. For other sets of values of $T,\;R,\;\mu_l$, the social welfare function may be unimodal in $\lambda$, or even if it is bimodal one of the two modes may persist as optimal. In these cases the optimal arrival rate is continuous and monotone in the corresponding parameter. For example Figure \ref{figsocT3ml}, shows the dependence on $\mu_l$ for $T=3$ and $R=5$, where there exists a unique equilibrium for any $\mu_l$, which is continuous and increasing in $\mu_l$, as is also the socially optimal rate. In this case, the equilibrium is always greater than the socially optimal rate, whenever they take non-zero 

\begin{figure}
\centering
\includegraphics[width=0.9\textwidth]{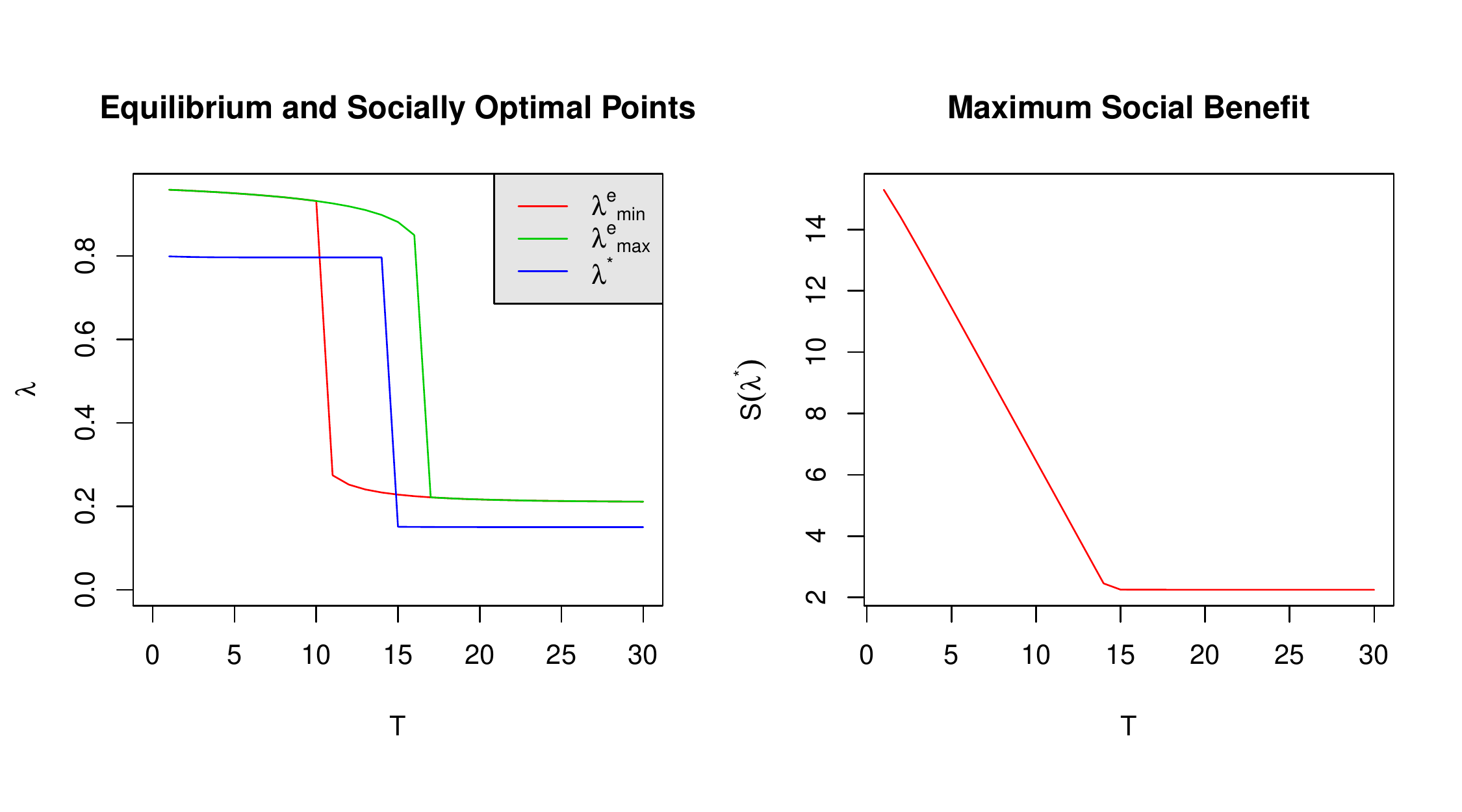}
\caption{Equilibrium and Socially Optimal rates as a function of $T$ for $R=25$, $\mu_l=\frac{1}{4}$.}\label{Tsoceq}
\end{figure}
\begin{figure}
\centering
\includegraphics[width=0.9\textwidth]{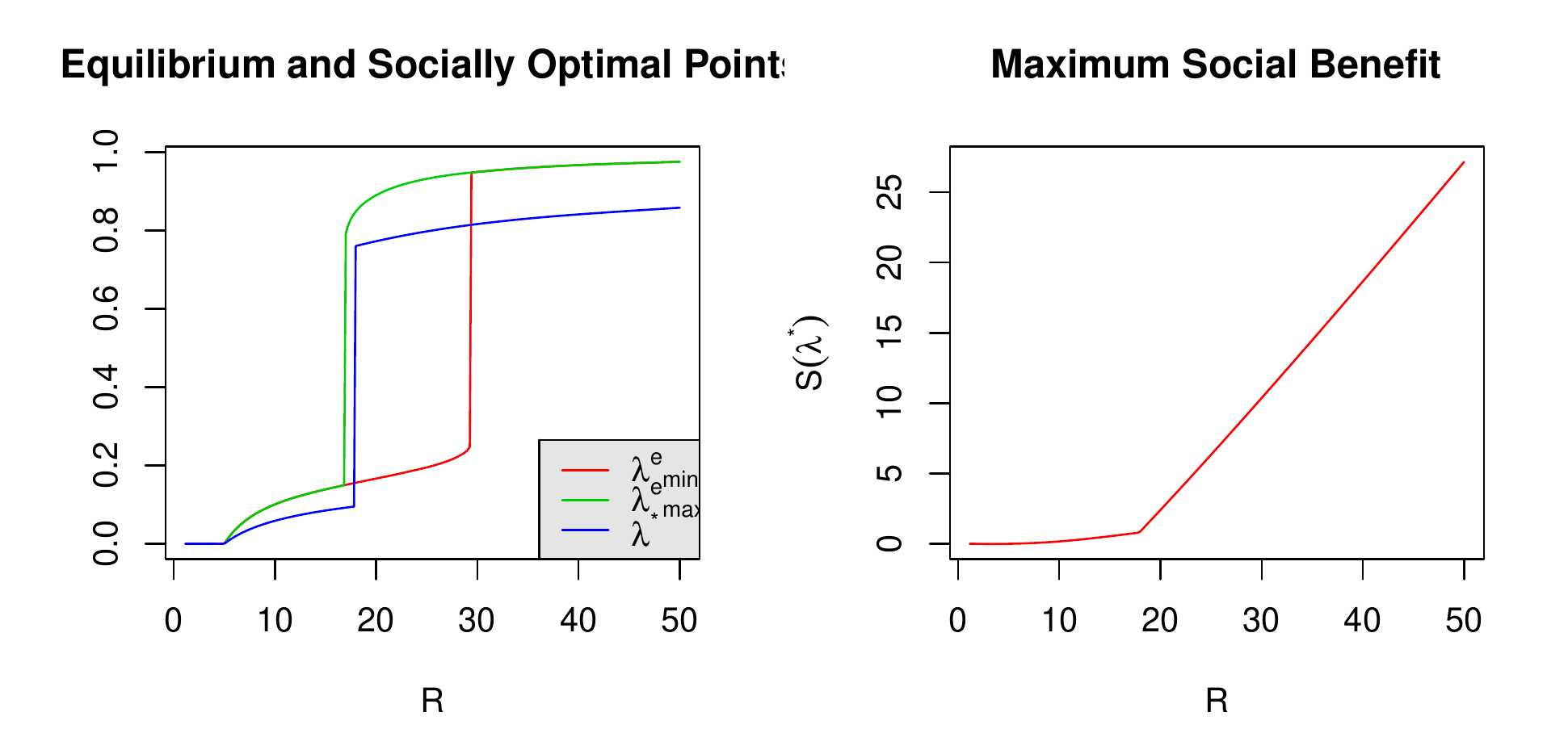}
\caption{Equilibrium and Socially Optimal rates as a function of $R$ for $T=10$,$\mu_l=\frac{1}{5}$.}\label{RT10soceq}
\end{figure}
\begin{figure}
\centering
\includegraphics[width=0.9\textwidth]{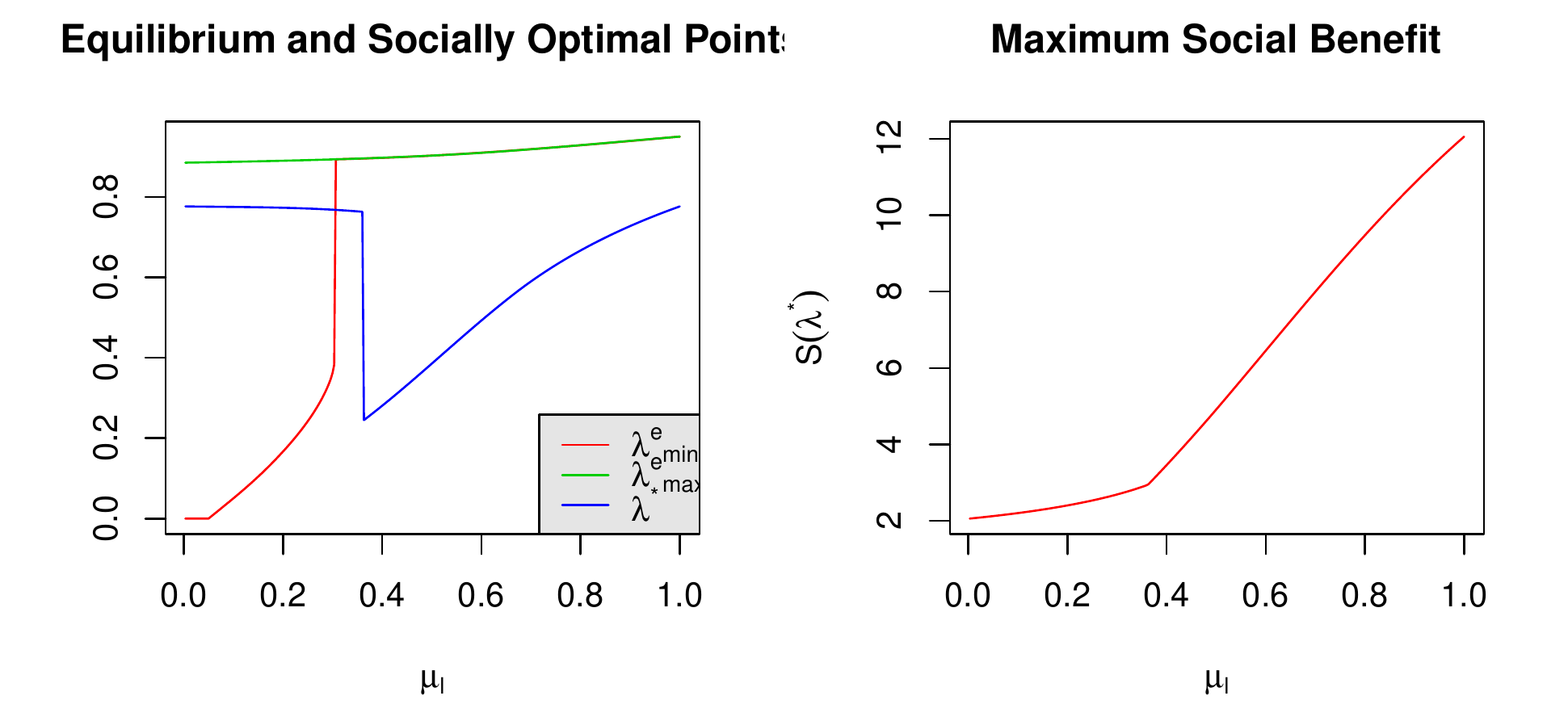}
\caption{Equilibrium and Socially Optimal rates as a function of $\mu_l$ for $T=10$, $R=20$.}\label{T10figsoc}
\end{figure}
\begin{figure}
\centering
\includegraphics[width=0.9\textwidth]{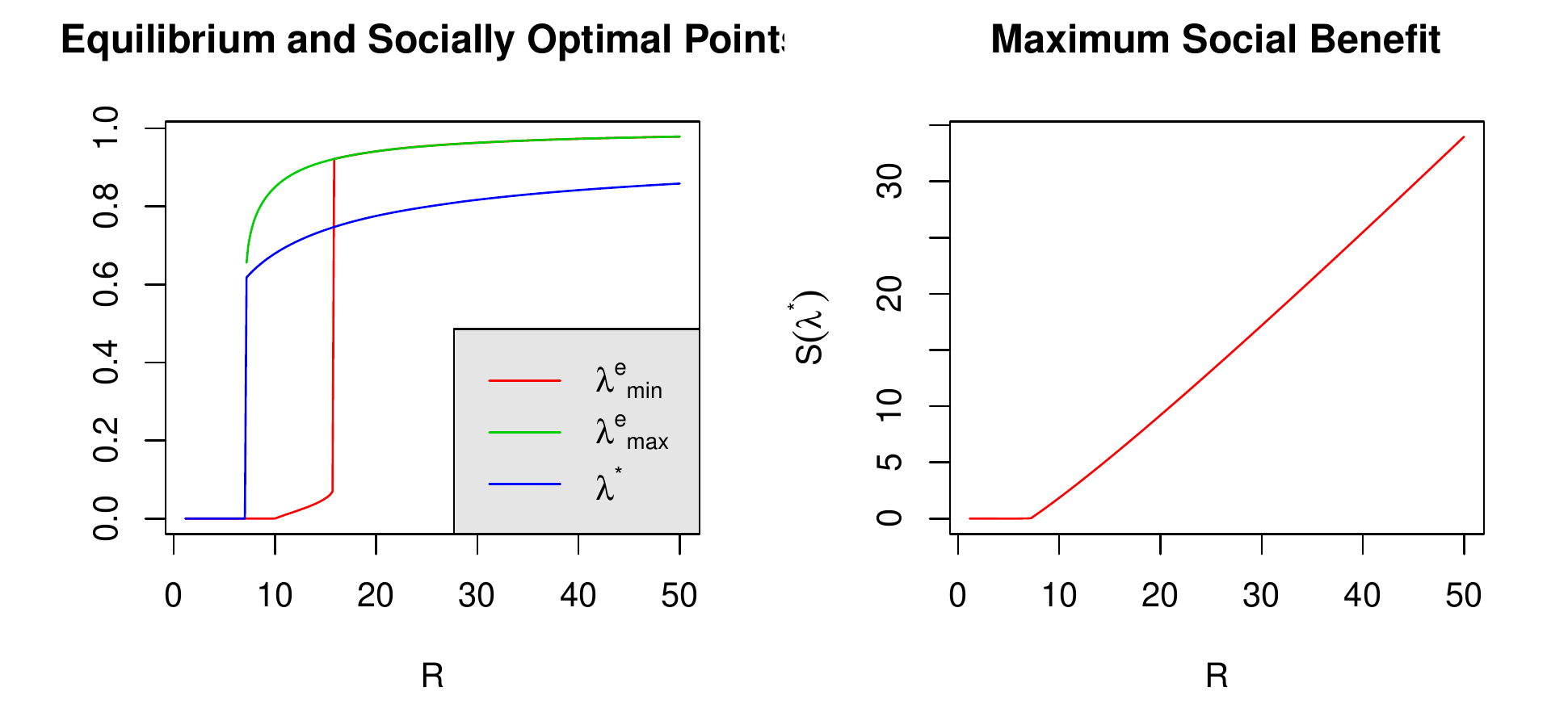}
\caption{Equilibrium and Socially Optimal rates as a function of $R$ for $T=3$, $\mu_l=\frac{1}{10}$.}\label{soceqRT3}
\end{figure}
\begin{figure}
\centering
\includegraphics[width=0.9\textwidth]{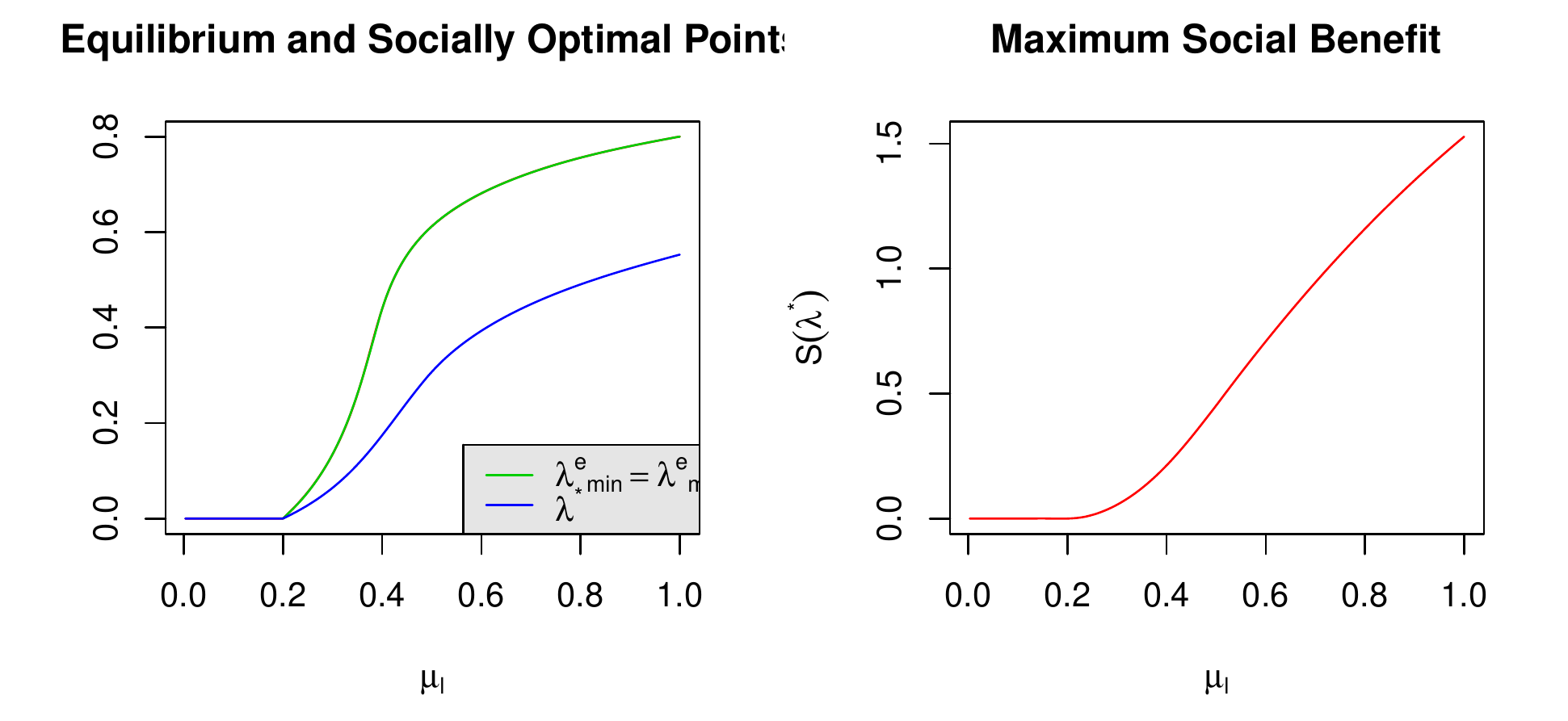}
\caption{Equilibrium and Socially Optimal rates as a function of $\mu_l$ for $T=3$, $R=5$.}\label{figsocT3ml}
\end{figure}

\section{The Case $T=1$}\label{sec-Special}

\noindent In this section, we analyze the case where the service threshold $T$ is set to $1$ and,
thus the service policy prescribes that the server uses the low
rate as long as the queue is empty and switches to the high rate as soon as a customer joins the queue.

 This policy may be appropriate when there is a large difference between the two service rate values, and an assurance must be given to the customers that when there is at least one person waiting in the queue the service is expedited.

 From the analysis point of view the expressions of the delay function is now considerably simpler, permitting more complete analytical results.

 As in the general model, we first characterize the symmetric equilibrium strategies and then proceed to social benefit maximization.

 For $T=1$, the delay function in \eqref{simexpwt} becomes
\begin{equation}\label{expwt_T1}
W_1(\lambda)=\frac{1}{(1-\lambda)[\mu_l+\lambda(1-\mu_l)]},\;\lambda\in[0,1).
\end{equation}
\noindent In the following lemma we present some properties of the delay function $W_1(\lambda)$. The proof follows from simple calculus and is omitted.
\begin{figure}
\centering
\includegraphics[width=0.9\textwidth]{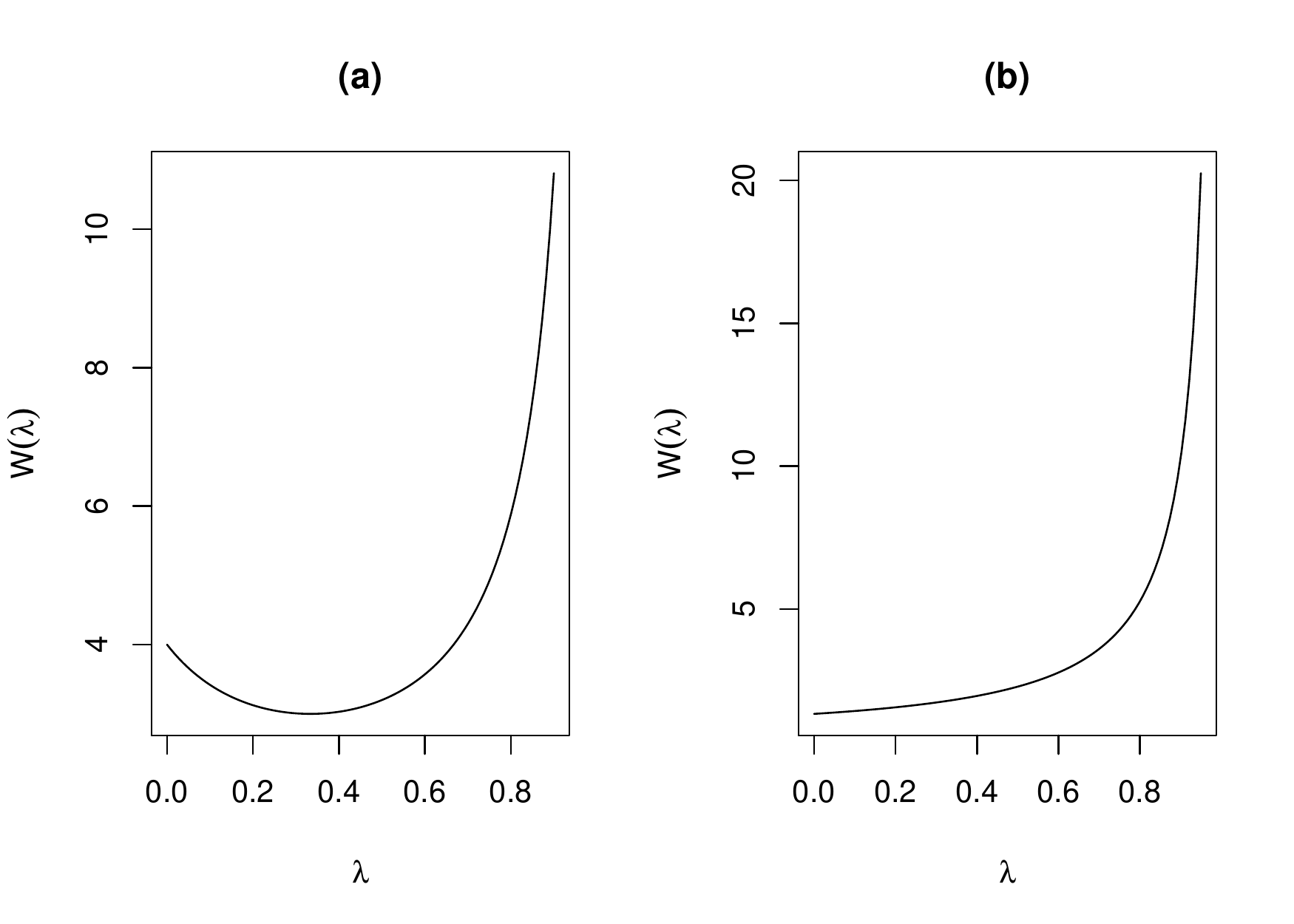}
\caption{The expected waiting time $W(\lambda)$ for $T=1$ and (a)$\mu_l<\frac{1}{2}$ or (b)$\mu_l>\frac{1}{2}$}\label{T1wplots}
\end{figure}
\begin{lemma}\label{ewtpropT1}
\begin{itemize}
\item[]
\item[i.] $W_1(0)=\frac{1}{\mu_l}$ and $\lim_{\lambda\rightarrow1}W_1(\lambda)=+\infty$
\item[ii.] If $\mu_l>\frac{1}{2}$, $W_1(\lambda)$ is increasing in $\lambda\in(0,1)$.
\item[iii.] If $\mu_l<\frac{1}{2}$, $W_1(\lambda)$ is decreasing
in $\lambda\in(0,\frac{1-2\mu_l}{2(1-\mu_l)})$ and is increasing in
$\lambda\in(\frac{1-2\mu_l}{2(1-\mu_l)},1)$.
\item[iv.] $W_1(\lambda)$ is convex in $\lambda\in(0,1)$.
\end{itemize}
\end{lemma}
\noindent The plot of $W_1(\lambda)$ for two representative values of $\mu_l$ is shown in Figure \ref{T1wplots}.\\
\vskip0.1cm
\noindent {\bf{Equilibrium Analysis}} For the equilibrium analysis we follow the same steps as in Section
\ref{sec-Equilibrium}. Now, equation \eqref{rationaleq} becomes quadratic:
\begin{equation}\label{poleqT1}
H_1(\lambda)=R(1-\mu_l)\lambda^2+R(2\mu_l-1)\lambda+1-R\mu_l=0.
\end{equation}
\noindent This equation $H_1(\lambda)=0$ can be solved explicitly and may have either one or two roots in $[0,1)$ (the case of zero roots is excluded because of Lemma \ref{leastroot}), depending on the values of $R$ and $\mu_l$.

 Performing a detailed but otherwise standard analysis of the roots, we obtain the following theorem on the number of symmetric equilibria.
\begin{theorem}\label{equilstrT1}
For $\Lambda\geq1$ the following hold.
\begin{itemize}
\item[i.] If $R<1$, then $0$ is the unique equilibrium.
\item[ii.] If $1\leq R\leq2$ and $\mu_l\leq\frac{1}{R}$, then $0$ is the unique equilibrium.
\item[iii.] If $R>2$ and $\mu_l<1-\frac{R}{4}$, then $0$ is the unique equilibrium.
\item[iv.] If $R>2$ and $\mu_l=1-\frac{R}{4}$, then the equilibrium arrival rates are $0$ and $\lambda_0^e$.
\item[v.] If $R>2$ and $1-\frac{R}{4}<\mu_l<\frac{1}{R}$, then the equilibrium arrival are $0,\;\lambda_1^e$ and $\lambda_2^e$.
\item[vi.] If $R\geq1$ and $\mu_l>\frac{1}{R}$, then $\lambda_1^e$ is the unique equilibrium.
\end{itemize}
\noindent where
\begin{eqnarray}
\lambda_1^e=\frac{R(1-2\mu_l)+\sqrt{R(R-4(1-\mu_l))}}{2R(1-\mu_l)},\\
\lambda_2^e=\frac{R(1-2\mu_l)-\sqrt{R(R-4(1-\mu_l))}}{2R(1-\mu_l)}\\
\end{eqnarray}
\noindent and
\begin{equation}
\lambda_0^e=\frac{(1-2\mu_l)}{2(1-\mu_l)}.
\end{equation}
\end{theorem}

\noindent In Figure \ref{equilstrfigT1} we present the number of symmetric equilibria in different regions of the $R-\mu_l$ plane,
corresponding to the cases listed in Theorem \ref{equilstrT1}.
\begin{figure}
\centering
\includegraphics[width=0.9\textwidth]{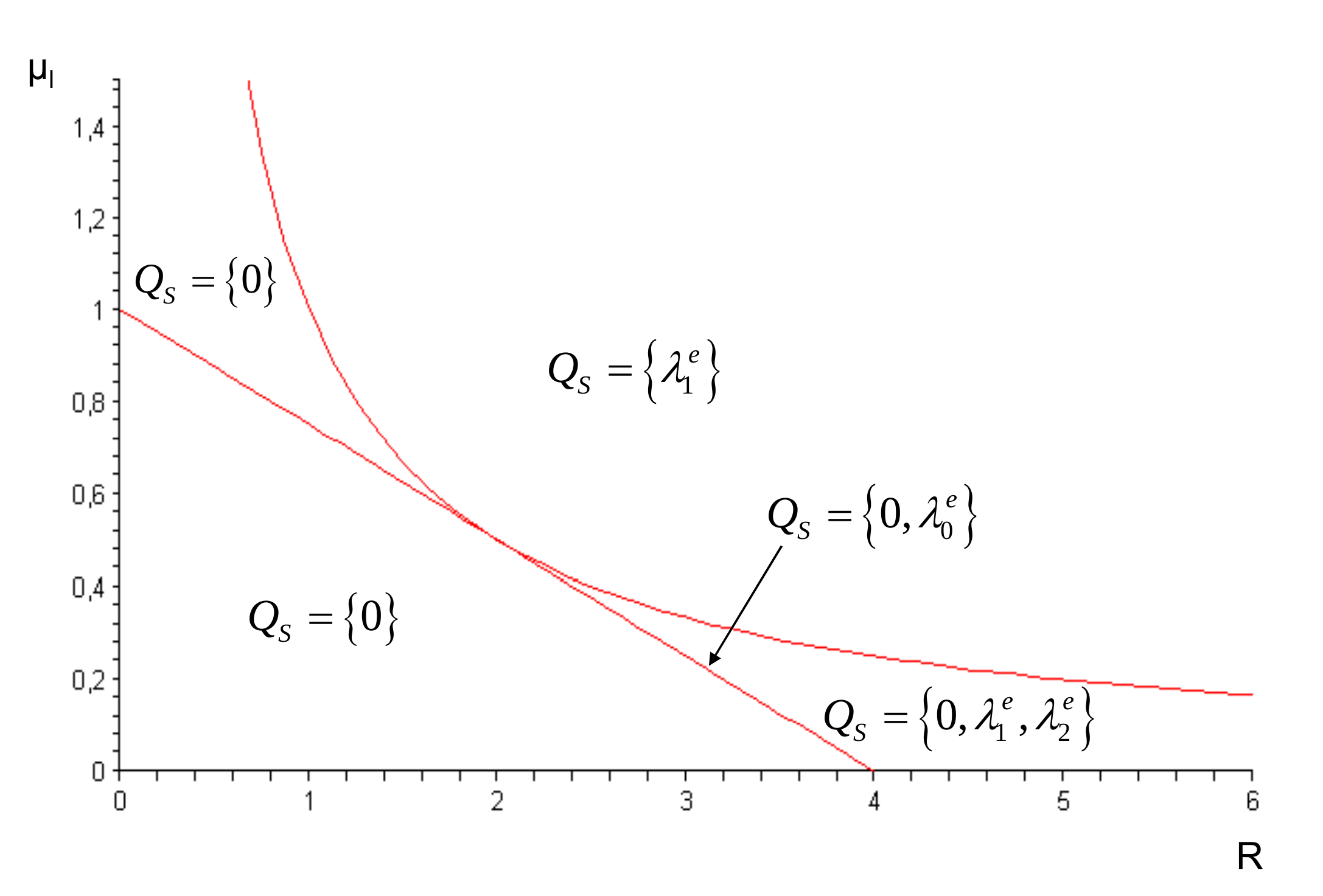}
\caption{Symmetric Equilibria for the Case T=1, for several values
of the parameters $\mu_l$ and $R$}\label{equilstrfigT1}
\end{figure}
\vskip0.1cm
\noindent {\bf{Social Optimization}} Proceeding to the social
optimization problem, we let $S_1(\lambda)$ be the social welfare
function in the case of $T=1$. Thus,
\begin{equation}\label{socfuncT1}
S_1(\lambda)=\lambda\left[R-W_1(\lambda)\right],
\end{equation}
\noindent In the following lemma we derive some
properties of $S_1(\lambda)$.
\begin{lemma}\label{propS1}
\begin{itemize}
\item[]
\item[i.] If $\mu_l<\frac{1}{2}$, then $S_1(\lambda)$ is strictly
concave in $\lambda\in[0,1)$.
\item[ii.] If $\mu_l=1-\frac{R}{4}<\frac{1}{2}$, then $S_1(\lambda)$ is
decreasing and concave in $\lambda\in(\lambda_0^e,1)$.
\item[iii.] If $1-\frac{R}{4}<\mu_l<\frac{1}{2}$, then $S_1(\lambda)$ is
decreasing and concave in $\lambda\in(\lambda_1^e,1)$.
\end{itemize}
\end{lemma}
\begin{proof}
\noindent The first two derivatives of $S_1(\lambda)$ are
\begin{eqnarray}
\label{1stderS1}S_1'(\lambda)=R-W_1(\lambda)-\lambda W_1'(\lambda)\\
\label{2ndderS1}S_1''(\lambda)=-2W_1'(\lambda)-\lambda
W_1''(\lambda).
\end{eqnarray}
\noindent {\bf{i.}} If $\mu>\frac{1}{2}$, $W_1(\lambda)$ is
increasing and convex in $\lambda\in[0,1)$, as stated in Lemma
\ref{ewtpropT1}. By \eqref{2ndderS1}, it follows that
$S_1''(\lambda)<0\;\forall\lambda\in[0,1)$, and
the statement holds.

 {\bf{ii.}} The expression of $S_1'(\lambda)$ in
\eqref{1stderS1} can be rewritten as
\begin{equation}\label{alt1stderS1}
S_1'(\lambda)=-\frac{H_1(\lambda)}{(1-\lambda)(\mu_l+(1-\mu_l)\lambda)}-\lambda
W_1'(\lambda)
\end{equation}
\noindent where $H_1(\lambda)$ is the polynomial in
\eqref{poleqT1}.

 Now, for $\mu_l=1-\frac{R}{4}<\frac{1}{2}$, we obtain that
$H_1(\lambda)>0,\;\forall\lambda\in[0,\lambda_0^e)\cup(\lambda_0^e,1)$.
Also, $W_1(\lambda)$ is increasing in $\lambda\in(\lambda_0^e,1)$.
By \eqref{alt1stderS1} and \eqref{2ndderS1} we obtain that
$S_1'(\lambda)<0$ and $S_1''(\lambda)<0$ for
$\lambda\in(\lambda_0^e,1)$. Thus, $S_1(\lambda)$ is decreasing and
concave in $\lambda\in(\lambda_0^e,1)$.

 {\bf{iii.}} Similarly to {\it{ii.}} it can be shown that $S_1(\lambda)$ is decreasing and concave in
$\lambda\in(\lambda_1^e,1)$.\
\end{proof}
\noindent In the following proposition we prove a result on the relationship between the optimal and equilibrium arrival rates. In this case, $\lambda^e$ refers to a positive and stable equilibrium arrival rate.
\begin{proposition}\label{socoptrate}
For any $\mu_l\in(0,1)$ and $R\geq1$ there exists an equilibrium $\lambda^e$ such that $\lambda^*\leq\lambda^e$ for any social optimal
arrival rate $\lambda^*\in[0,1)$.
\end{proposition}
\begin{proof}
\noindent In order to prove the statement of Proposition
\ref{socoptrate} we consider the
cases of Lemma \ref{propS1} for the value of $\mu_l$.

 First, assume $\mu_l>\frac{1}{2}$. By
the concavity of the social welfare function $S_1(\lambda)$, as
stated in Lemma \ref{propS1}, it follows that there exists a unique
social optimal arrival rate $\lambda^*$, which maximizes
$S_1(\lambda)$.

 In addition, if $\mu_l\leq\frac{1}{R}$ the unique
equilibrium arrival rate is $\lambda^e=0$ with
$S_1'(0)=R-W(0)=R-\frac{1}{\mu_l}\leq0$. Thus, $S_1'(0)\leq
S_1'(\lambda^*)$ and by the monotonicity of $S_1'(\lambda)$ we
obtain that $\lambda^*\leq0$, which results to $\lambda^*=0$. If, on
the contrary, $\mu_l>\frac{1}{R}$, the unique equilibrium arrival
rate is $\lambda^e=\lambda^e_1$ with $S_1'(\lambda^e_1)<0$. Once
again, by the monotonicity of $S_1'(\lambda)$ we prove that
$\lambda^*\leq\lambda^e$ for any $\lambda^*$.

 Next, we consider the case where
$\mu_l=1-\frac{R}{4}<\frac{1}{2}$. By Lemma \ref{propS1}, we obtain
that $S_1(\lambda)$ is decreasing and concave in
$\lambda\in(\lambda_0^e,1)$. Thus, $\lambda^*\leq\lambda_0^e$ for any $\lambda^*$, where $\lambda_0^e$ is the stable
and positive equilibrium arrival rate, as stated in Proposition
\ref{equilstrT1}.

 Finally, for the case of $1-\frac{R}{4}<\mu_l<\frac{1}{2}$,
it follows similarly that $\lambda^*\leq\lambda_1^e$ for any optimal arrival rate
$\lambda^*$, where
$\lambda_1^e$ is the stable positive equilibrium arrival rate.\
\end{proof}

\section{Conclusions and Extensions}
\label{sec-Summary}
\noindent In this paper, we considered the problem of customer
equilibrium joining behavior in an unobservable $M|M|1$ queue with
dynamically adjusted service rate according to a threshold type
service policy, known to all customers.  We proved that there exist
multiple equilibria, three at most, due to the presence of both
positive and negative externalities for different values of the
effective arrival rate. For the social welfare optimization problem,
we explored the relationship between the optimal and the equilibrium
arrival rates in several numerical experiments for different values of
system parameters. We demonstrated that in cases where there exist
multiple equilibria, the socially optimal point lies between the
extreme equilibrium values, whereas in all other cases it is always
less than the equilibrium arrival rate.

This work can be extended in several directions. First, it would be of
interest to determine the equilibrium strategies under different
levels of information on system congestion or the status of the server
upon customer arrival. Some preliminary work has shown that the
equilibrium analysis in these problems is significantly more complex,
and non-threshold equilibria may exist.  In terms of the form of the
service control policy, one might consider an alternative where
instead of increasing the service speed with congestion, one or more
standby servers are activated. Such a policy is more relevant in
situations where the service is provided by human workers.  Other
extensions include generally distributed service requirements and
dynamic service speed, series of queues with different parameters
and/or service policies, and systems where the specific service policy
threshold is not announced to the customers.

\bibliographystyle{agsm}

\end{document}